\documentclass[reqno]{amsproc}
\usepackage[utf8]{inputenc}
\usepackage{amsmath,amsfonts,amssymb}
\usepackage{enumitem}
\usepackage{hyperref}
\usepackage{etoolbox}
\usepackage{stmaryrd}
\hypersetup{
  hidelinks,
  pdftitle={Multiply warped products with harmonic Weyl tensor and zero radial Weyl curvature},
  pdfauthor={M. A. R. M. Hor\'acio},
  pdfsubject={Differential geometry},
  pdfkeywords={Multiply warped products, harmonic Weyl tensor, radial Weyl curvature}
}

\newtheorem{theorem}{\bf Theorem}
\newtheorem{proposition}{\bf Proposition}
\newtheorem{corollary}{\bf Corollary}
\newtheorem{lemma}{\bf Lemma}
\theoremstyle{definition}
\newtheorem{remark}{\bf Remark}

\numberwithin{equation}{section}
\patchcmd{\subequations}
  {\theparentequation\alph{equation}}
  {\theparentequation.\alph{equation}}{}{}

\everymath{\displaystyle}

\begin{document}

\title[Multiply warped products with harmonic Weyl tensor]
{Multiply warped products with harmonic Weyl tensor and zero radial Weyl curvature}

\author{M. A. R. M. Hor\'acio}
\address{Department of Mathematics, Universidade de Bras\'ilia, Bras\'ilia, DF, Brazil}
\email{matheus.andrade5488@gmail.com}
\thanks{Corresponding author. ORCID: \href{https://orcid.org/0009-0001-5658-7467}{0009-0001-5658-7467}.}


\begin{abstract}
In this article, we study multiply warped products with a one-dimensional base, harmonic Weyl tensor and zero radial Weyl curvature. We prove that these manifolds have at most three fibers and describe their local geometry. The proof uses a conformal change which makes the warping functions quadratic polynomials, together with the compatibility equations imposed by the Ricci tensor. The resulting metrics are either obtained from locally conformally flat multiply warped products by replacing their space-form fibers with Einstein manifolds, or belong to a three-fiber family with equal fiber dimensions and positive Einstein constants. Examples show that the bound is sharp and that neither Weyl condition implies the other. We also discuss the stronger restrictions arising from soliton-type equations.
\end{abstract}

\maketitle

\section{Introduction and main results}

We consider a multiply warped product
\begin{align}\label{MWPS}
M^n=I\times_{h_1}N_1^{r_1}\times\cdots\times_{h_k}N_k^{r_k},
\qquad n=1+\sum_{i=1}^k r_i\geq4,
\end{align}
endowed with the Riemannian metric
\begin{align}\label{metricMWPS}
g=\mathrm{d}s^2+\sum_{i=1}^k h_i(s)^2g_{N_i}.
\end{align}
Here $I$ is a connected open interval, each $N_i$ is connected, and each $h_i$ is positive and smooth. We write $E_1=\partial_s$.

\begin{remark}[Number of fibers]\label{counting}
If two warping functions differ by a positive constant factor on $I$, we rescale the corresponding fiber metrics and combine them into a single fiber. Throughout the article, $k$ is counted after these identifications. We keep the interval base fixed, so a constant warping function, when present, is retained and counted. There is at most one such function. This is the only difference from the convention in \cite[Remark 2.4]{BHSalmost}, where constant warping functions are absorbed into the base.
\end{remark}

An estimate on the number of fibers is a first step towards a local description of these manifolds. In the locally conformally flat case, Brozos-V\'azquez, Garc\'ia-R\'io and V\'azquez-Lorenzo \cite[Theorem 9]{brovaz2005}, \cite[Theorem 3.7]{brovaz} showed that there are at most three fibers and described the possible local models. Harmonic Weyl curvature is a natural condition to consider beyond local conformal flatness, since it includes both locally conformally flat and Einstein metrics. The harmonic-curvature examples of Derdzi\'nski and Piccione \cite{derdPiccione}, however, show that harmonicity alone allows an arbitrary number of fibers.

Soliton-type equations give stronger estimates. Kim obtained local classifications of gradient Ricci solitons with harmonic Weyl curvature in dimension four \cite{kimfour} and in higher dimensions \cite{kimall}. For gradient Ricci and gradient almost Ricci solitons with harmonic Weyl tensor, the logarithmic derivatives in a potential-adapted multiply warped representation satisfy a nonzero polynomial of degree at most two, giving at most two fibers; see \cite{BHSricci} and \cite[Lemma 3.10 and Corollary 3.11]{BHSalmost}.

Here a representation is \emph{adapted to the potential} $f$ when $f=f(s)$ and $f'\neq0$: the slices $s=\mathrm{constant}$ are then the connected regular level sets of $f$, and $\nabla f=f'E_1$. The fiber equations relate their warping functions to the Ricci eigenvalues transverse to $\nabla f$, as we explain in Section \ref{additional}. Under local conformal flatness, the regular potential levels are totally umbilical and a single fiber occurs \cite[Corollary 1.4]{catino}. Catino \cite[Theorem 1.1]{catino} obtained the corresponding single-fiber description for generalized quasi-Einstein manifolds by combining harmonic Weyl curvature with radial Weyl vanishing.

The main goal of this article is to study the corresponding pair of Weyl conditions on a multiply warped product itself. The condition of zero radial Weyl curvature is
\begin{align}\label{radialW}
W(E_1,\cdot,\cdot,\cdot)=0.
\end{align}
For \eqref{metricMWPS}, this is equivalent to $W(E_1,U,E_1,V)=0$ for all vectors tangent to the fibers.

The role of \eqref{radialW} is explained by the conformal transformation law of the Cotton tensor, displayed in \eqref{conformalCotton}. For a metric with harmonic Weyl tensor, it is precisely the condition which preserves harmonicity under every conformal change depending only on the base. This permits the conformal reduction used by Brozos-V\'azquez, Garc\'ia-R\'io and V\'azquez-Lorenzo in the locally conformally flat setting \cite[Theorem 9]{brovaz2005}, \cite[Lemma 3.1]{brovaz}: dividing by the square of one warping function produces a direct product factor. Catino's proof uses the same Cotton transformation to justify a conformal change determined by the potential. Here the conformal factor is chosen from the warping functions themselves.

Our main result is the following.

\enlargethispage{-2pt}
\begin{theorem}\label{main}
Let $M^n$ be a multiply warped product as in \eqref{MWPS}, with fibers counted as in Remark \ref{counting}. Suppose that its Weyl tensor is harmonic and satisfies \eqref{radialW}. Then every fiber is Einstein and $k\leq3$. Moreover, exactly one of the following alternatives holds:
\begin{enumerate}[label=\textup{(\roman*)}]
\item\label{purecase} The Weyl tensor has nonzero components only within individual fibers. Locally, the metric is obtained from a locally conformally flat multiply warped product by replacing its space-form fibers with Einstein manifolds of the same dimensions and Ricci constants.
\item\label{exceptionalcase} There are three fibers of the same dimension $r\geq2$. After a local change of base coordinate, constant rescaling and permutation of the fiber metrics, the metric is
\begin{align}\label{exceptional}
g=e^{2\varphi(t)}\bigl(\mathrm{d}t^2+t^2g_{N_1}
 +(1-t)^2g_{N_2}+t^2(1-t)^2g_{N_3}\bigr),
\end{align}
where $\varphi$ is smooth,
\begin{align}\label{exceptionalRic}
\mathrm{Ric}_{N_i}=(2r-1)g_{N_i},\qquad i\in\{1,2,3\},
\end{align}
and $t$ varies in an interval avoiding $0$ and $1$. All three Weyl sectional components between distinct fibers are nonzero.
\end{enumerate}
Conversely, the metrics described in \textup{(i)} and \textup{(ii)} have harmonic Weyl tensor and zero radial Weyl curvature.
\end{theorem}

The positive warping functions in \eqref{exceptional} are obtained by taking absolute values of the polynomial factors. The local alternatives in Theorem \ref{main} are compatible along the connected interval, and the number of fibers is counted on the whole interval $I$.

The first step in the proof is to use the curvature assumptions to show that the fibers are Einstein and that all warping functions satisfy a common linear equation of order three. We then divide the metric by the square of one of its warping functions. This produces a constant warping function and makes the radial Schouten eigenvalue constant. A further conformal change makes this eigenvalue zero, and the common differential equation becomes $h_i'''=0$. The fiber Ricci equations then provide the compatibility relations which bound the number of fibers.

The bound remains sharp when all warping derivatives are required to be nowhere zero.

\begin{theorem}\label{sharpness}
For every integer $r\geq2$, there is a multiply warped product over an interval with three Einstein fibers of dimension $r$, harmonic Weyl tensor and zero radial Weyl curvature. The three warping derivatives are nowhere zero, their logarithmic derivatives are pairwise distinct at every point, and the Weyl sectional components between distinct fibers are nonzero and have both signs.
\end{theorem}

In Section \ref{examples}, we prove Theorem \ref{sharpness} and give examples showing that neither Weyl condition implies the other. Finally, Section \ref{additional} relates the number of fibers to the transverse Ricci spectrum and gives an alternative proof of Catino's single-fiber result in the potential-adapted multiply warped setting.

\section{Preliminaries}\label{prelim}

We adopt the curvature convention
\[
\mathrm{Rm}(X,Y,Z,T)
 =\left\langle\nabla_Y\nabla_X Z-\nabla_X\nabla_Y Z
       +\nabla_{[X,Y]}Z,T\right\rangle.
\]
We use the Schouten tensor with the normalization
\begin{align}\label{schouten}
\mathcal{A}=\mathrm{Ric}-\frac{R}{2(n-1)}g.
\end{align}
Thus the Weyl and Cotton tensors are
\begin{align*}
W&=\mathrm{Rm}-\frac{1}{n-2}\mathcal{A}\varowedge g,\\
C(X,Y,Z)&=(\nabla_X\mathcal{A})(Y,Z)-(\nabla_Y\mathcal{A})(X,Z),
\end{align*}
where $\varowedge$ denotes the Kulkarni--Nomizu product. We take $\delta W$ to be the negative contraction of $\nabla W$ in its derivative and first curvature arguments. Since $n\geq4$, harmonicity of $W$ is equivalent to $C=0$.

For the metric \eqref{metricMWPS}, set
\begin{align}\label{xidef}
\xi_i=\frac{h_i'}{h_i},\qquad S=\sum_{i=1}^k r_i\xi_i.
\end{align}
A prime denotes differentiation with respect to the arclength coordinate on the base. The Levi--Civita connection, for basic lifts $U_i,V_i$ tangent to $N_i$, is given by
\begin{align*}
\nabla_{E_1}E_1&=0,&
\nabla_{E_1}U_i=\nabla_{U_i}E_1&=\xi_iU_i,\\
\nabla_{U_i}V_i&=\nabla^{N_i}_{U_i}V_i-\xi_i g(U_i,V_i)E_1,&
\nabla_{U_i}V_j&=0\quad(i\neq j).
\end{align*}
The standard curvature formulas \cite[Propositions 2.2--2.6]{DobUn} consequently give
\begin{align}\label{ricformulas}
\mathrm{Ric}(E_1,E_1)&=-S'-\sum_{i=1}^k r_i\xi_i^2,\nonumber\\
\mathrm{Ric}|_{TN_i}&=\mathrm{Ric}_{N_i}-(\xi_i'+S\xi_i)g|_{TN_i},\\
R&=\sum_{i=1}^k\frac{R_{N_i}}{h_i^2}-2S'-S^2-\sum_{i=1}^k r_i\xi_i^2.\nonumber
\end{align}
All mixed Ricci components vanish. For unit vectors $U_i\in TN_i$ and $U_j\in TN_j$, with $i\neq j$, we also have
\begin{align}\label{sectional}
\mathrm{Rm}(E_1,U_i,E_1,U_i)=-\xi_i'-\xi_i^2,
\qquad
\mathrm{Rm}(U_i,U_j,U_i,U_j)=-\xi_i\xi_j.
\end{align}
The components with exactly one radial argument vanish.

\begin{lemma}\label{einsteinODE}
Under the hypotheses of Theorem \ref{main}, there are constants $\mu_i$ such that
\[
\mathrm{Ric}_{N_i}=\mu_i g_{N_i}.
\]
Write $\lambda_1=\mathrm{Ric}(E_1,E_1)$ and denote by $\lambda_{i+1}$ the Ricci eigenvalue on $TN_i$. Then
\begin{align}
\xi_i'+\xi_i^2
 &=-\frac{1}{n-2}\left(\lambda_1+\lambda_{i+1}-\frac{R}{n-1}\right),\label{radialODE}\\
\lambda_{i+1}'-\frac{R'}{2(n-1)}
 &=(\lambda_1-\lambda_{i+1})\xi_i,\label{cottonODE}\\
\frac{\mu_i}{h_i^2}&=\lambda_{i+1}+\xi_i'+S\xi_i.\label{compatRic}
\end{align}
In particular, all warping functions satisfy the same linear equation
\begin{align}\label{thirdODE}
h_i'''
 +\frac{2}{n-2}\left(\lambda_1-\frac{R}{2(n-1)}\right)h_i'
 +\frac{1}{n-2}\left(\lambda_1'-\frac{R'}{2(n-1)}\right)h_i=0.
\end{align}
\end{lemma}

\begin{proof}
Since $C=0$, the connection formulas and \eqref{ricformulas} give
\begin{align}\label{fiberCotton}
0=C(E_1,U_i,E_1)=\frac{U_i(R)}{2(n-1)}
 =\frac{U_i(R_{N_i})}{2(n-1)h_i^2}.
\end{align}
Hence each $R_{N_i}$ is constant. The trace-free part within $TN_i$ of the radial Weyl equation, together with \eqref{ricformulas} and \eqref{sectional}, yields
\[
\mathrm{Ric}_{N_i}=\frac{R_{N_i}}{r_i}g_{N_i}.
\]
This proves the first assertion, with $\mu_i=R_{N_i}/r_i$.

The radial Weyl equation gives \eqref{radialODE}. From the connection formulas we obtain
\begin{align}\label{cottoncomponent}
C(E_1,U_i,V_i)=
\left(\lambda_{i+1}'-\frac{R'}{2(n-1)}
       -(\lambda_1-\lambda_{i+1})\xi_i\right)g(U_i,V_i),
\end{align}
which proves \eqref{cottonODE}. The other Cotton components vanish because the fiber metrics are Einstein. Equation \eqref{compatRic} is the fiber component of \eqref{ricformulas}. Finally, differentiating
\[
h_i''=-\frac{1}{n-2}\left(\lambda_1+\lambda_{i+1}-\frac{R}{n-1}\right)h_i
\]
and using \eqref{cottonODE} gives \eqref{thirdODE}.
\end{proof}

\begin{remark}\label{continuation}
If $h_i=ch_j$ on a nonempty open subinterval, then $h_i-ch_j$ solves \eqref{thirdODE} with zero initial data through order two. Uniqueness implies $h_i=ch_j$ on all of $I$. Consequently, the number of fibers does not change upon restriction to an open subinterval. This allows us to obtain the estimate on $I$ from the local calculations below.
\end{remark}

\section{Conformal changes and quadratic warping functions}\label{conformal}

The conformal transformation of the Cotton tensor explains the radial hypothesis. For $\widetilde g=e^{2\varphi}g$ and an arbitrary smooth function $\varphi$, the formula in the conventions of Section \ref{prelim} is
\begin{align}\label{conformalCotton}
\widetilde C(X,Y,Z)
 =C(X,Y,Z)+(n-2)W(\nabla\varphi,Z,X,Y).
\end{align}
Here the gradient and the Weyl tensor on the right-hand side are computed with respect to $g$, and both Cotton tensors are viewed as covariant three-tensors; see \cite[Appendix]{catino} and \cite[Section 2]{goverNagy}. In particular, when $\varphi=\varphi(s)$, the additional term is $(n-2)\varphi'W(E_1,Z,X,Y)$. Thus \eqref{radialW} makes the Cotton tensor unchanged by every such conformal change. Conversely, if $C=0$ and every conformal change depending on $s$ has zero Cotton tensor, choosing $\varphi$ with $\varphi'\neq0$ gives \eqref{radialW}. We use this freedom to choose a representative suited to the fiber equations.

\begin{lemma}\label{conformallemma}
Suppose that \eqref{metricMWPS} has zero radial Weyl curvature. For every smooth $\varphi:I\to\mathbb{R}$, the metric $\widetilde g=e^{2\varphi}g$ also has zero radial Weyl curvature and has the same Cotton tensor as a covariant three-tensor. In particular, harmonicity of the Weyl tensor is preserved. In the arclength coordinate $u$, defined by $\mathrm{d}u=e^{\varphi}\mathrm{d}s$, its warping functions and logarithmic derivatives are
\begin{align}\label{conformalwarps}
\widetilde h_i=e^{\varphi}h_i,\qquad
\widetilde\xi_i=e^{-\varphi}(\xi_i+\varphi').
\end{align}
Moreover, for $\widetilde E_1=e^{-\varphi}E_1$,
\begin{align}\label{conformalSchouten}
\widetilde{\mathcal A}(\widetilde E_1,\widetilde E_1)
 =e^{-2\varphi}\left(\mathcal A(E_1,E_1)
 -(n-2)\left(\varphi''-\frac{(\varphi')^2}{2}\right)\right).
\end{align}
\end{lemma}

\begin{proof}
The covariant Weyl tensor satisfies $\widetilde W=e^{2\varphi}W$, while \eqref{conformalCotton} gives $\widetilde C=C$ under the radial hypothesis. The expressions in \eqref{conformalwarps} follow from the change of arclength. Finally, the Schouten transformation, with the normalization \eqref{schouten}, is
\[
\widetilde{\mathcal A}
 =\mathcal A-(n-2)\left(\nabla^2\varphi
   -\mathrm{d}\varphi\otimes\mathrm{d}\varphi
   +\frac{|\nabla\varphi|^2}{2}g\right).
\]
Its radial component is \eqref{conformalSchouten}.
\end{proof}

\begin{proposition}\label{quadratic}
Every metric in Theorem \ref{main} is locally conformal, by a factor depending on the base, to a multiply warped product whose radial Schouten eigenvalue is zero and whose warping functions are polynomials of degree at most two.

For this representative, write the metric and its warping functions as
\begin{align}\label{polynomialmetric}
g=\mathrm{d}t^2+\sum_{i=1}^k h_i(t)^2g_{N_i},
\qquad h_i(t)=A_i t^2+B_i t+C_i.
\end{align}
Define the constants
\begin{align}\label{Kdef}
K_{ij}=h_i'h_j'-h_i''h_j-h_ih_j''
      =B_iB_j-2A_iC_j-2A_jC_i.
\end{align}
Then the fiber equations are equivalent to
\begin{align}\label{rationalcompat}
\sum_{j=1}^k\frac{r_jK_{ij}}{h_j}
 =\frac{\mu_i+K_{ii}}{h_i},\qquad i\in\{1,\ldots,k\}.
\end{align}
Conversely, Einstein fibers with coefficients $\mu_i$ and quadratic warping functions satisfying \eqref{rationalcompat}, on an interval where the warping functions do not vanish, determine a metric with zero radial Schouten eigenvalue and both Weyl properties. For unit vectors tangent to distinct fibers,
\begin{align}\label{mixedW}
W(U_i,U_j,U_i,U_j)=-\frac{K_{ij}}{h_i h_j}.
\end{align}
\end{proposition}

\begin{proof}
We keep the fiber metrics $g_{N_i}$ fixed and distinguish the two conformal metrics used in the normalization. First, set
\begin{align}\label{firstconformalmetric}
\widetilde g=h_1(s)^{-2}g
 =\frac{\mathrm{d}s^2}{h_1(s)^2}+g_{N_1}
  +\sum_{i=2}^k\frac{h_i(s)^2}{h_1(s)^2}g_{N_i}.
\end{align}
Since $h_1>0$, we may introduce an arclength coordinate $u$ by
\[
\mathrm{d}u=\frac{\mathrm{d}s}{h_1(s)},
\qquad
\widetilde h_i(u)=\frac{h_i(s(u))}{h_1(s(u))}.
\]
Thus the first conformal metric takes the form
\begin{align}\label{firstproductmetric}
\widetilde g=\mathrm{d}u^2+g_{N_1}
 +\sum_{i=2}^k\widetilde h_i(u)^2g_{N_i},
\qquad \widetilde h_1=1.
\end{align}
Its unit radial vector is $\widetilde E_1=\partial_u=h_1E_1$, and Lemma \ref{conformallemma} shows that it has harmonic Weyl tensor and zero radial Weyl curvature.

Let $\widetilde{\mathcal A}$ and $\widetilde R$ denote the Schouten tensor and scalar curvature of $\widetilde g$. The first fiber is now a direct product factor, so its Ricci eigenvalue is $\mu_1$. Applying \eqref{radialODE} and \eqref{cottonODE} to $\widetilde g$, with $\widetilde\xi_1=0$, gives
\[
\widetilde{\mathcal A}(\widetilde E_1,\widetilde E_1)
 =-\mu_1+\frac{\widetilde R}{2(n-1)},
\qquad
\frac{\mathrm{d}}{\mathrm{d}u}
 \left(\mu_1-\frac{\widetilde R}{2(n-1)}\right)=0.
\]
Since $\mu_1$ is constant, both $\widetilde R$ and the normalized radial Schouten eigenvalue
\[
a=\frac{\widetilde{\mathcal A}(\widetilde E_1,\widetilde E_1)}{n-2}
\]
are constant. We choose the additive constant in $u$ so that the point under consideration has coordinate $u=0$.

For a positive function $\rho=\rho(u)$, consider the second conformal metric
\begin{align}\label{secondconformalmetric}
\widehat g=\rho(u)^{-2}\widetilde g
 =\frac{\mathrm{d}u^2}{\rho(u)^2}
  +\sum_{i=1}^k\frac{\widetilde h_i(u)^2}{\rho(u)^2}g_{N_i}.
\end{align}
Define its arclength coordinate $t$ and warping functions by
\[
\mathrm{d}t=\frac{\mathrm{d}u}{\rho(u)},
\qquad
\widehat h_i(t)=\frac{\widetilde h_i(u(t))}{\rho(u(t))}.
\]
Then
\begin{align}\label{secondproductmetric}
\widehat g=\mathrm{d}t^2+\sum_{i=1}^k\widehat h_i(t)^2g_{N_i},
\qquad \widehat E_1=\partial_t=\rho\widetilde E_1.
\end{align}
Writing $\widehat{\mathcal A}$ for the Schouten tensor of $\widehat g$, formula \eqref{conformalSchouten} with $\varphi=-\log\rho$ gives
\[
\frac{\widehat{\mathcal A}(\widehat E_1,\widehat E_1)}{n-2}
 =a\rho^2+\rho\rho''-\frac{(\rho')^2}{2},
\]
where the derivatives of $\rho$ are taken with respect to $u$. We choose $\rho$ so that this expression vanishes. More explicitly, for $a\neq0$ put $c=\sqrt{|a|/2}$ and take
\[
\begin{array}{c|c|c}
 &\rho(u)&t(u)\\ \hline
 a>0&\cos^2(cu)&c^{-1}\tan(cu)\\
 a=0&1&u\\
 a<0&\cosh^2(cu)&c^{-1}\tanh(cu).
\end{array}
\]
Each choice has $\rho(0)=1$. After restricting to a neighborhood on which $\rho>0$, the coordinate $t$ is well defined and
\begin{align}\label{normalization}
\frac{\mathrm{d}t}{\mathrm{d}u}=\rho^{-1},\qquad
 a\rho^2+\rho\rho''-\frac{(\rho')^2}{2}=0.
\end{align}
Thus $\widehat{\mathcal A}(\widehat E_1,\widehat E_1)=0$, while Lemma \ref{conformallemma} again preserves both Weyl conditions. If $a=0$, this second change is the identity. In terms of the original metric and coordinate, the two changes together are
\begin{align}\label{combinedconformalmetric}
\widehat g=\frac{g}{h_1(s)^2\rho(u(s))^2},
\qquad
\mathrm{d}t=\frac{\mathrm{d}s}{h_1(s)\rho(u(s))}.
\end{align}

Equation \eqref{thirdODE}, applied to $\widehat g$ in the coordinate $t$, now reads
\[
\frac{\mathrm{d}^3\widehat h_i}{\mathrm{d}t^3}=0.
\]
Hence $\widehat h_i(t)=A_it^2+B_it+C_i$, proving the asserted conformal representation. For the rest of the proof, we write $g$ and $h_i$ for $\widehat g$ and $\widehat h_i$, as in \eqref{polynomialmetric}. All curvature quantities now refer to this final metric, $E_1=\partial_t$, and primes denote differentiation with respect to $t$. The fiber metrics have not changed, so their Einstein constants remain $\mu_i$.

The vanishing radial Schouten eigenvalue and \eqref{radialODE} give
\begin{align}\label{quadraticRic}
\lambda_1=\frac{R}{2(n-1)}=-\sum_{j=1}^k r_j\frac{h_j''}{h_j},
\qquad
\lambda_{i+1}=\lambda_1-(n-2)\frac{h_i''}{h_i}.
\end{align}
Consequently,
\begin{align*}
h_i\sum_{j=1}^k\frac{r_jK_{ij}}{h_j}
 &=h_i h_i'S-(n-1)h_i h_i''+\lambda_1h_i^2\\
 &=\mu_i+(h_i')^2-2h_i h_i''
 =\mu_i+K_{ii},
\end{align*}
where the second equality follows from \eqref{compatRic}. This proves \eqref{rationalcompat}, and \eqref{mixedW} follows from \eqref{sectional} and \eqref{quadraticRic}.

For the converse, compute the Ricci tensor of the proposed metric directly from \eqref{ricformulas}. Equation \eqref{rationalcompat} gives
\[
\mathrm{Ric}|_{TN_i}
 =\left(\mathrm{Ric}(E_1,E_1)-(n-2)\frac{h_i''}{h_i}\right)g|_{TN_i}.
\]
Taking the trace, and using $\mathrm{Ric}(E_1,E_1)=-\sum_jr_jh_j''/h_j$, yields
\[
R=2(n-1)\mathrm{Ric}(E_1,E_1).
\]
Thus the radial Schouten eigenvalue is zero and \eqref{quadraticRic} holds for the actual Ricci tensor. The radial Weyl condition follows. The tangential Schouten eigenvalues are $-(n-2)h_i''/h_i$, whose derivatives are $(n-2)h_i''h_i'/h_i^2$ because $h_i'''=0$. Equation \eqref{cottoncomponent} therefore gives $C=0$.
\end{proof}

\section{The number of fibers and the local geometry}\label{bound}

We now use the compatibility equations to bound the number of fibers. We first consider constant functions and polynomials with a repeated root. The remaining cases are distinguished by whether two polynomials have a common root. Whenever we clear denominators, the resulting polynomial identity holds identically, since it holds on an open interval. We may therefore compare coefficients or evaluate it at a root, independently of whether that root belongs to the interval on which the metric is defined.

\begin{proof}[\bf Proof of Theorem \ref{main}]
The Einstein assertion follows from Lemma \ref{einsteinODE}. Work with the representative of Proposition \ref{quadratic}. By Remark \ref{continuation}, its warping polynomials are pairwise nonproportional.

\smallskip
\noindent\textit{A constant warping function or a repeated root.}
Suppose first that $K_{ii}=0$ for some $i$. By \eqref{Kdef}, either $h_i$ is constant or $h_i=A_i(t-c)^2$ for a real number $c$. In the latter case, the change
\begin{align}\label{inversion}
u=\frac{1}{t-c},\qquad
\widetilde g=(t-c)^{-4}g,\qquad
\widetilde h_j(u)=u^2h_j(c+u^{-1})
\end{align}
makes the selected warping function constant. The new warping functions are still quadratic, their radial Schouten eigenvalue is zero, and their constants in \eqref{Kdef} remain $K_{ij}$. The coordinate change is defined on the regular interval because $h_i>0$ there.

We may therefore assume that one warping function is constant. By \eqref{quadraticRic} and the Ricci equation on that fiber, $\lambda_1$ is constant. On the other hand,
\[
\lambda_1=-\sum_{j=1}^k r_j\frac{h_j''}{h_j}
 =-\sum_{A_j\neq0}\frac{2r_jA_j}{h_j}.
\]
Suppose that at least one $A_j$ is nonzero. Dividing each quadratic polynomial by its leading coefficient and clearing denominators gives
\begin{align}\label{degreecomparison}
\lambda_1\prod_{A_j\neq0}\frac{h_j}{A_j}
 =-2\sum_{A_j\neq0}r_j
   \prod_{\substack{A_\ell\neq0\\\ell\neq j}}\frac{h_\ell}{A_\ell}.
\end{align}
Each factor $h_j/A_j$ is monic of degree two. If $\lambda_1\neq0$, the left-hand side has degree two more than every term on the right, which is impossible. Thus $\lambda_1=0$. However, all products on the right have the same degree and leading coefficient $1$, so the leading coefficient of their weighted sum is $-2\sum_{A_j\neq0}r_j\neq0$. This is again a contradiction. Hence every $A_j$ vanishes, and \eqref{quadraticRic} shows that all Ricci eigenvalues vanish.

For a nonconstant affine warping function, \eqref{compatRic} becomes
\begin{align}\label{affinecompat}
\mu_j+(h_j')^2=S h_j'h_j.
\end{align}
If there is any such warping function, the polynomial $\prod_jh_j^{r_j}$ is nonconstant, so its logarithmic derivative $S$ is not identically zero. Since $h_j'$ is a nonzero constant and $h_j>0$ on the interval, \eqref{affinecompat} implies that $\mu_j+(h_j')^2$ is nonzero. For any two nonconstant warping functions, division of their equations \eqref{affinecompat} therefore shows that their ratio is constant. Hence there are at most two fibers, with at most one nonconstant warping function, and all the components in \eqref{mixedW} vanish.

\smallskip
\noindent\textit{Warping functions with no common root.}
We may now suppose that $K_{ii}\neq0$ for every $i$. Each $h_i$ is nonconstant and has no repeated root. Suppose that no two of these polynomials have a common root in $\mathbb{C}$. Fix $i\neq j$, multiply \eqref{rationalcompat} by $h_j$ and evaluate at a root of $h_j$. All terms vanish except $r_jK_{ij}$, because none of the other denominators vanishes there. Hence $K_{ij}=0$.

These equalities also show that the polynomials $h_i$ are linearly independent. Indeed, if $\sum_i b_i h_i=0$, the bilinearity of the expression in \eqref{Kdef} gives $\sum_i b_iK_{ij}=0$ for each $j$. Since $K_{ij}=0$ for $i\neq j$ and $K_{jj}\neq0$, we obtain $b_j=0$. The vector space of polynomials of degree at most two has dimension three, and therefore $k\leq3$. Again, all components in \eqref{mixedW} vanish.

\smallskip
\noindent\textit{Two warping functions with a common root.}
Otherwise, two nonproportional polynomials have a common root. This root is real, since two real polynomials of degree at most two sharing a nonreal root also share its conjugate. It is simple because $K_{ii}\neq0$. The inversion \eqref{inversion} at this root makes both warping functions affine. Their roots are distinct, so an affine change of coordinate, together with a constant conformal rescaling and constant rescalings of the fiber metrics, gives
\[
h_1=t,\qquad h_2=1-t.
\]
Here and below we may use signed polynomial factors, since the metric contains their squares. Each polynomial has a fixed sign on the interval under consideration.

Put $c_i=\mu_i+1$ for $i\in\{1,2\}$. Since $h_1''=h_2''=0$, equations \eqref{compatRic} and \eqref{quadraticRic} give
\[
\lambda_1t^2+St=c_1,\qquad
\lambda_1(t-1)^2+S(t-1)=c_2.
\]
Solving these equations, we obtain
\begin{align}\label{commonroot}
\lambda_1=\frac{c_1}{t}-\frac{c_2}{t-1},
\qquad S=\frac{(1-t)c_1}{t}+\frac{tc_2}{t-1}.
\end{align}
Let $c$ be a root of one of the polynomials. Since every root is simple, the identity $S=\sum_jr_jh_j'/h_j$ gives
\[
\lim_{t\to c}(t-c)S(t)=\sum_{h_j(c)=0}r_j>0.
\]
Indeed, a polynomial vanishing at $c$ contributes $r_j$ to the limit, while each of the other polynomials contributes zero. If $c\notin\{0,1\}$, the expression for $S$ in \eqref{commonroot} is regular at $c$, so the same limit would be zero. Thus every root is either $0$ or $1$. The polynomials are nonconstant, have no repeated root and have degree at most two. Hence each is proportional to one of
\begin{align}\label{threepolynomials}
t,\qquad 1-t,\qquad t(1-t).
\end{align}
This proves $k\leq3$ in the remaining case.

To determine the geometry, set $r_3=0$ if the third polynomial is absent, and otherwise normalize it as in \eqref{threepolynomials}. The first two fiber equations now read
\begin{align}\label{dimensioncompat}
\mu_1&=r_1+r_3-1+(r_3-r_2)\frac{t}{1-t},\nonumber\\
\mu_2&=r_2+r_3-1+(r_3-r_1)\frac{1-t}{t}.
\end{align}
The left-hand sides are constant, so $r_1=r_2=r_3=:r>0$. In particular, the third fiber is present. The third equation then gives
\[
\mu_1=\mu_2=\mu_3=2r-1.
\]
Since a one-dimensional fiber has zero Ricci curvature, $r\geq2$. Restoring the conformal factor gives \eqref{exceptional}.

Conversely, the three polynomials in \eqref{threepolynomials} have
\begin{align}\label{exceptionalK}
K_{11}=K_{22}=K_{33}=1,\qquad
K_{12}=-1,\qquad K_{13}=K_{23}=1.
\end{align}
With $r_i=r$ and $\mu_i=2r-1$, they satisfy \eqref{rationalcompat}. Proposition \ref{quadratic} and Lemma \ref{conformallemma} prove that every metric \eqref{exceptional} has both Weyl properties. Equations \eqref{mixedW} and \eqref{exceptionalK} show that its components between distinct fibers are nonzero.

It remains to identify the geometry when all components between distinct fibers vanish. The curvature formulas show that the remaining Weyl components are supported within individual fibers. Replace each Einstein fiber of dimension $r_i\geq2$ by a local space form of sectional curvature $\mu_i/(r_i-1)$, keeping the warping functions fixed; replace a one-dimensional fiber by a flat line. The Ricci and Schouten tensors are unchanged. The only possible Weyl components of the resulting metric are now the components within each space-form fiber, where the curvature expression is a scalar multiple of the constant-curvature tensor. The trace-free identity for $W$ forces this scalar to vanish. Hence the resulting metric is locally conformally flat, proving \textup{(i)}.

Conversely, replacing the space-form fibers of a locally conformally flat metric by Einstein fibers with the same dimensions and Ricci constants preserves its Ricci and Schouten formulas, its Cotton tensor and its radial Weyl components. This proves the converse in \textup{(i)}.

Finally, in the original coordinate the quantities
\begin{align}\label{originalK}
h_i'h_j'+\frac{\lambda_{i+1}+\lambda_{j+1}-R/(n-1)}{n-2}h_i h_j
\end{align}
are constant, as follows by differentiating and applying \eqref{radialODE} and \eqref{cottonODE}. For $i\neq j$, they equal $-h_i h_jW(U_i,U_j,U_i,U_j)$ and are unchanged by radial conformal changes. Thus the alternatives cannot change along $I$. Together with Remark \ref{continuation}, this proves the fiber count on $I$ and the stated local classification.
\end{proof}

\begin{corollary}\label{lowfibers}
A metric in Theorem \ref{main} with one or two fibers is locally obtained by Einstein replacement of locally conformally flat data. If any Weyl sectional component between distinct fibers is nonzero, then $n=3r+1$ for an integer $r\geq2$.
\end{corollary}

\begin{remark}\label{horizontalnullity}
A natural extension of \eqref{radialW} to a multiply warped product $B\times_{h_1}N_1\times\cdots\times_{h_k}N_k$ would be to require
\[
W(X,\cdot,\cdot,\cdot)=0
\quad\text{for every vector field $X$ tangent to $B$.}
\]
For a Riemannian metric with $\dim B\geq2$ and $\dim M\geq4$, this condition already implies $C=0$. Indeed, the horizontal curvature equations make the scalar curvature depend only on the base, and the mixed equations make the fibers Einstein. Moreover, $\nabla W$ vanishes whenever two of its curvature arguments are horizontal. The contracted Bianchi identity therefore gives the vanishing of the Cotton components with at least two horizontal entries. The differential Bianchi identity, applied to two orthogonal horizontal directions, gives the vanishing of those with one horizontal entry. The vertical components vanish by the Einstein condition on the fibers. Thus harmonic Weyl curvature is automatic under this extension, rather than an independent additional hypothesis.
\end{remark}

\section{Examples and the two Weyl conditions}\label{examples}

\subsection{Sharpness of the bound}

\begin{proof}[\bf Proof of Theorem \ref{sharpness}]
Choose Einstein manifolds $(N_i^r,g_{N_i})$ satisfying \eqref{exceptionalRic}, for instance suitably rescaled round spheres. In \eqref{exceptional}, take $\varphi=0$ and $t\in(0,1/2)$. By Theorem \ref{main}, the resulting metric has both Weyl properties and three fibers. Its warping functions have derivatives $1$, $-1$ and $1-2t$, all nowhere zero on this interval. Moreover,
\[
\xi_1=\frac1t,\qquad
\xi_2=-\frac1{1-t},\qquad
\xi_3=\frac1t-\frac1{1-t},
\]
so $\xi_2<\xi_3<\xi_1$. Finally, \eqref{mixedW} and \eqref{exceptionalK} give
\begin{align*}
W(U_1,U_2,U_1,U_2)&=\frac1{t(1-t)},\\
W(U_1,U_3,U_1,U_3)&=-\frac1{t^2(1-t)},\qquad
W(U_2,U_3,U_2,U_3)=-\frac1{t(1-t)^2},
\end{align*}
for unit vectors $U_i$ tangent to the indicated fibers. This proves the assertion.
\end{proof}

\subsection{Each Weyl condition separately}

\begin{proposition}\label{independence}
Harmonicity of the Weyl tensor and zero radial Weyl curvature do not imply one another, even for multiply warped products with flat fibers.
\end{proposition}

\begin{proof}
The harmonic-curvature metrics of Derdzi\'nski and Piccione \cite[Section 5]{derdPiccione} include multiply warped products with arbitrarily many flat fibers. For $n\geq4$, harmonic curvature is equivalent to harmonic Weyl curvature together with constant scalar curvature. Thus their examples have harmonic Weyl tensor, while Theorem \ref{main} implies that radial Weyl curvature cannot vanish identically when there are at least four fibers.

For the converse, take two flat fibers $N_1^r,N_2^r$, with $r\geq2$, and consider
\begin{align}\label{radialcounterexample}
g=\mathrm{d}s^2+e^{2s}g_{N_1}+e^{-2s}g_{N_2}.
\end{align}
Here $n=2r+1$, and \eqref{ricformulas} gives
\[
(\xi_1,\xi_2)=(1,-1),\qquad
(\lambda_1,\lambda_2,\lambda_3)=(-2r,0,0),\qquad R=-2r.
\]
Consequently, for unit $U_i\in TN_i$,
\[
W(E_1,U_i,E_1,U_i)=-1-\frac{-2r+1}{2r-1}=0.
\]
The remaining radial components also vanish by the multiply warped curvature formulas. On the other hand, \eqref{cottoncomponent} yields
\[
C(E_1,U_1,U_1)=2r,\qquad C(E_1,U_2,U_2)=-2r.
\]
Hence the Cotton tensor is nonzero.
\end{proof}

\section{Soliton-type equations}\label{additional}

We consider the generalized quasi-Einstein equation
\begin{align}\label{GQE}
\mathrm{Ric}+\nabla^2f-\mu\,\mathrm{d}f\otimes\mathrm{d}f=\lambda g.
\end{align}
This is the equation studied by Catino \cite{catino}. It includes gradient Ricci solitons when $\mu=0$ and $\lambda$ is constant, and gradient almost Ricci solitons when $\mu=0$ and $\lambda$ is allowed to vary, as introduced by Pigola, Rigoli, Rimoldi and Setti \cite{pigola}.

Let the product representation be adapted to $f$, so that $f=f(s)$ and $f'\neq0$. Its radial direction is therefore normal to the regular potential levels. For vectors $U_i,V_i$ tangent to $N_i$, the connection formulas give
\[
\nabla^2f(U_i,V_i)=f'\xi_i g(U_i,V_i).
\]
Since $\mathrm{d}f$ vanishes on these vectors, the fiber components of \eqref{GQE} yield
\begin{align}\label{solitonspectrum}
\lambda_{i+1}=\lambda-f'\xi_i,
\qquad
\lambda_{i+1}-\lambda_{j+1}=-f'(\xi_i-\xi_j).
\end{align}
The transverse Ricci eigenvalues are those of the restriction of the Ricci endomorphism to $(\nabla f)^\perp$, that is, to the tangent space of a potential level. Since $f'\neq0$, \eqref{solitonspectrum} identifies their distinct values at each point with the distinct values of the $\xi_i$. A representation with $k$ fibers therefore has at most $k$ transverse Ricci eigenvalues. Proportional warping functions give the same eigenvalue, while different fibers may have coincident eigenvalues at a point. For gradient Ricci and gradient almost Ricci solitons with harmonic Weyl tensor, the two-fiber estimate thus gives at most two transverse Ricci eigenvalues; compare \cite{kimall,BHSricci} and \cite[Remark 3.8]{BHSalmost}.

Catino \cite[Theorem 1.1]{catino} proved the single-fiber local description under harmonic Weyl curvature and zero radial Weyl curvature. His proof uses the conformal metric $e^{-2f/(n-2)}g$ and the geometry of the eigendistributions of a Codazzi tensor, studied by Derdzi\'nski \cite{derdCodazzi}. The following proposition records Catino's conclusion in the present setting. We give an alternative proof using the fiber equations, with the nondegeneracy condition stated explicitly.

\begin{proposition}[Catino \cite{catino}]\label{solitoncomparison}
Suppose that a metric in Theorem \ref{main} satisfies \eqref{GQE}, where $f$, $\mu$ and $\lambda$ are smooth functions of $s$, and
\begin{align}\label{nondegenerate}
f'\neq0,\qquad 1+(n-2)\mu\neq0
\end{align}
on $I$. Then $k=1$, and the regular level sets of $f$ are totally umbilical. If $W=0$, the fiber has constant sectional curvature.
\end{proposition}

\begin{proof}[Alternative proof]
The radial component of \eqref{GQE} is
\[
\lambda_1=\lambda-f''+\mu(f')^2.
\]
Using \eqref{solitonspectrum} and subtracting the equations \eqref{radialODE} for two fibers, we obtain
\begin{align}\label{solitonradialdiff}
(\xi_i'+\xi_i^2)-(\xi_j'+\xi_j^2)
 =\frac{f'}{n-2}(\xi_i-\xi_j).
\end{align}
On the other hand, substitution in \eqref{cottonODE}, followed by subtraction for $i$ and $j$, yields
\begin{align}\label{solitoncottondiff}
f'\bigl((\xi_i'+\xi_i^2)-(\xi_j'+\xi_j^2)\bigr)
 =-\mu(f')^2(\xi_i-\xi_j).
\end{align}
Combining these two identities, we obtain
\[
\bigl(1+(n-2)\mu\bigr)(f')^2(\xi_i-\xi_j)=0.
\]
The assumptions imply $\xi_i=\xi_j$ for all $i,j$. Integrating $h_i'/h_i=h_j'/h_j$ shows that the ratio of any two warping functions is constant, so $k=1$ by Remark \ref{counting}.

The second fundamental form of a slice, with normal $E_1$, is $g(\nabla_XE_1,Y)=\xi_1g(X,Y)$. Since $f'\neq0$, these slices are precisely the connected regular potential levels, proving total umbilicity. Finally, for a single Einstein fiber the tangential curvature formula, together with $W=0$, makes the sectional curvature of the fiber constant.
\end{proof}

In particular, the conclusion applies to gradient Ricci and gradient almost Ricci solitons with the two Weyl conditions on their regular potential-adapted product regions. In the terminology above, they have a single transverse Ricci eigenvalue. When the Weyl tensor vanishes, the fiber has constant sectional curvature, as in \cite[Corollary 1.4]{catino}. The same factor $1+(n-2)\mu$ controls the difference between the radial and transverse Ricci eigenvalues of $e^{-2f/(n-2)}g$ in Catino's conformal argument. Thus the three-fiber bound concerns a larger class than these soliton models.

\section*{Acknowledgments}

M. A. R. M. Hor\'acio was supported by CAPES Finance Code 001. He thanks the Mathematics Department of Universidade de Bras\'ilia, where this work was conducted. He is particularly grateful to his doctoral advisor, Jo\~ao Paulo dos Santos, for his valuable guidance, under which the problem studied in this article arose.

\end{document}